\documentclass[preprint,12pt]{elsarticle}

\usepackage{amssymb}
 \usepackage{amsthm}
 \usepackage{amsmath}
\usepackage{geometry}

\usepackage[mathscr]{eucal}
\usepackage{color}
\usepackage{amsmath}
\usepackage{amssymb}
\usepackage{amsfonts}
\usepackage{amscd}
\usepackage{amsthm}
\usepackage{latexsym}
\usepackage{graphicx}

\usepackage{amsthm}

\def\be{\begin{equation}}
\def\ee{\end{equation}}
\def\bse{\begin{subequations}}
\def\ese{\end{subequations}}

\newtheorem{thm}{Theorem}
\newtheorem{lem}[thm]{Lemma}
\newtheorem{prop}{Proposition}

\newdefinition{rmk}{Remark}
\begin{document}

\begin{frontmatter}

\title{Mass conservative reaction diffusion systems describing cell polarity}
\author{Evangelos Latos\footnote{Institut f\"ur Mathematik und
    Wissenschaftliches Rechnen, Heinrichstra{\ss}e 36, 8010
    Graz, Austria. Email: \texttt{evangelos.latos@uni-graz.at}. }, 
     Takashi Suzuki\footnote{Center for Mathematical Modeling and Data Science, Osaka University. Email: \texttt{suzuki@sigmath.es.osaka-u.ac.jp}}
        }

\begin{abstract}
A reaction-diffusion system with mass conservation modelling cell polarity is considered. A range of the parameters is found where the solution converges exponentially to the constant equilibrium and the $\omega$-limit set of the solution is spatially homogeneous, containing the constant stationary solution as well as possible spatially homogeneous orbits.
\end{abstract}

\begin{keyword}
Global dynamics \sep Chemical reaction diffusion system \sep Fix-Caginalp equation \sep total mass conservation system.

\end{keyword}

\end{frontmatter}


\section{Introduction}

The present work studies the following mass conserved reaction-diffusion system 
\begin{eqnarray} 
& & u_t=D\Delta u+f(u,v), \nonumber\\ 
& & \tau v_t=\Delta v-f(u,v),\qquad \qquad \mbox{in $\Omega\times(0,T),$} \nonumber\\ 
& & \frac{\partial}{\partial \nu}(u,v)=0, \qquad \qquad \quad \quad \ \ \mbox{on $\partial\Omega\times(0,T),$} \nonumber\\ 
& & \left. (u,v)\right\vert_{t=0}=(u_0(x), v_0(x)), \quad \mbox{in $\Omega$} ,
 \label{rd0} 
\end{eqnarray} 
where $\Omega\subset{\bf R}^N$ is a bounded domain with smooth boundary $\partial\Omega$, $\nu$ is the outer unit normal vector, $D, \tau$ are positive constants, and $(u_0, v_0)=(u_0(x), v_0(x))\geq 0$, $(u_0, v_0)\not\equiv 0$, are the non-negative, non-trivial initial values, taken to be sufficiently smooth.

Given the sufficiently smooth nonlinearity $f=f(u,v)$, standard theory allows the existence of a unique  local-in-time classical solution $(u,v)=(u(\cdot, t), v(\cdot,t))$ to (\ref{rd0}), as it can be seen in \cite{ls13,lms15,mor12,mo10}.  The solution $(u,v)$ has the following total mass conservation property, 
\begin{equation} 
\frac{d}{dt}\int_\Omega u+\tau v\ dx=0. 
 \label{tmc}
\end{equation} 
The class of models, that we are going to study, were proposed in \cite{oi07} to describe cell polarity. The proposed mechanism shall separate different spieces inside the cell according to their diffusion coefficients, i.e. slow and fast diffusions shall localize the spieces near the membrane and in the cytosol, respectively. Three kind of molecules, are interacting. Each one of them has two phases, active and inactive which are characterized by slow and fast diffusions, respectively. The model problem (\ref{rd0}) focuses on these two phases of a single species, ignoring the interactions between the other species.

The model shall allow Turing pattern \cite{tur52}, which is the appearance of spatially inhomogeneous stable stationary states induced by diffusion. In \cite{oi07} the authors suggest  the following three models for this purpose,  
\begin{eqnarray} 
& & f(u,v)=-\frac{au}{u^2+b}+v, \nonumber\\ 
& & f(u,v)=-\alpha_1\left[\frac{ u+v}{\biggl(\alpha_2(u+v)+1\biggr)^2}-v\right], \nonumber\\ 
& & f(u,v)=\alpha_1(u+v)[(\alpha u+v)(u+v)-\alpha_2]  ,
 \label{models}
\end{eqnarray} 
where $a$, $b$, $\alpha,\alpha_1$, and $\alpha_2$ are positive constants.  

In \cite{mjek} the authors suggested

\begin{equation} 
f(u,v)=b_1v\left[\frac{\gamma u^2}{K^2+u^2}+k_0\right]-\delta u 
 \label{4th}
\end{equation} 
where $b_1,\gamma,\delta,k_0,K$ positive constants (see \cite{mjek}), system (\ref{rd0}) with the above reaction term will be referred in the following as the fourth model and is the main topic of study in this paper. The results of this paper can be directly applied to more general reaction terms of the type:
\[
f(u,v)=b_1v\left[\frac{\gamma u^\beta}{K^2+u^\beta}+k_0\right]-\delta u
\]
with $\beta>1$.

Main characteristics of the system are the following:
\begin{itemize}
\item Quasi positivity for $f(u,v)$ provides positivity for $(u,v)$: 
$$
f=f(u,v):\overline{\mathbb{R}}^2_+\to\mathbb{R}, \ \text{locally Lipschitz continuous with} \ f(0,v)\geq0\geq f(u,0)$$
Therefore, the solution is nonnegative, provided that nonnegative initial data are given.
\item 	Mass conservative reaction-diffusion system: $\displaystyle\frac{d}{dt}\int_\Omega u+\tau v \;dx=0\Rightarrow$
\begin{equation} \frac{1}{\vert\Omega\vert}\int_\Omega u+\tau v \;dx=\lambda=\frac{1}{\vert\Omega\vert}\int_\Omega u_0+\tau v_0 \;dx
 \label{totalmass}
\end{equation} 
\end{itemize}

For the global existence but also for uniform-in-time bounds of nonnegative classical solutions to this system in all space dimension we refer the interested reader to Theorem 1.1 in \cite{ftm}. Actually, the authors in \cite{ftm} consider an even more general class of systems where the reaction terms might have a (slightly super-)quadratic growth.

In this paper we give an answer to the natural question which rises next about the asymptotic behaviour of the solution and wether it converges to the equilibrium. So, is the solution to this 4th model (\ref{4th}) asymptotically spatially homogeneous or do we have a Turing paradigm (stable non-constant stationary state under the local enhancement and long-range inhibition)?

This work is organised as follows: in Section 2 we summarise what has been done in the previous relevant models. In Section 3 we present and prove some of the key features of the fourth model and we state our main Theorem \ref{Theorem}.
In Section 4 we prove our main result.

\section{Review of the previous work}
In the first and the second model, the stationary state is described by the elliptic eigenvalue problem with nonlocal term, with the eigenvalue associated with the total mass that is conserved in time. The stationary state has a variational functional $J$, while there is a Lyapunov function $L(u,v)$ for the non-stationary problem. This Lyapunov functional is reduced to the stationary variational functional, if the total mass of $(u,v)$ is prescribed. This remarkable structure, called semi-unfolding minimality, induces dynamical stability of the local minimizer of $J$. We will briefly revisit what has already been done for these models. 

\vspace{2mm} 

{\bf First model.} If we let 
\[ f(u,v)=h(u)+kv, \quad h(u)=-\frac{au}{u^2+b}, \ k=1, \] 
 the first model takes the form 
\begin{eqnarray} 
& & u_t=D\Delta u+h(u)+kv, \nonumber\\ 
& & \tau v_t=\Delta v-h(u)-kv, \qquad \ \ \mbox{in $\Omega\times(0,T)$}, \nonumber\\ 
& & \frac{\partial}{\partial \nu}(u,v)=0, \qquad \qquad \qquad \ \ \mbox{on $\partial\Omega\times(0,T)$}, \nonumber\\ 
& & \left. (u, v)\right\vert_{t=0}=(u_0(x), v_0(x)), \quad \mbox{in $\Omega$}.
 \label{rdm1a} 
\end{eqnarray} 

Henceforth, $C_i$, $i=1,2,\cdots, 9$ denote positive constants independent of $t$. Since this $h=h(u)$ is a smooth function of $u\in {\bf R}$ satisfying  
\begin{equation} 
h(0)=0\geq h(u) \geq -C_1, \quad u\geq 0,  
 \label{eqn:hu0}
\end{equation} 
if $0\leq (u_0,v_0)=(u_0(x), v_0(x))\in X=C^2(\overline{\Omega})^2$, then problem (\ref{rdm1a}) admits a unique classical solution $(u,v)=(u(\cdot,t), v(\cdot,t))$ uniformly bounded, and global-in-time (Theorem 1.1 in \cite{ftm}). Therefore, the orbit ${\cal O}=\{ (u(\cdot,t), v(\cdot,t))\}_{t\geq 0}$ is compact in $X$ and hence the $\omega$-limit set defined by 
\begin{equation} 
\omega(u_0, w_0)=\{ (u_\ast, w_\ast) \mid \mbox{$\exists t_k\uparrow +\infty $ s.t. $\Vert u(\cdot,t_k)-u_\ast, w(\cdot,t_k)-w_\ast\Vert_X=0$}\} 
 \label{omega-limit}
\end{equation} 
is nonempty, compact, and connected.
  
With
\[ w=Du+v,\quad \xi =1-\tau D, \] 
the system (\ref{rdm1a}) transforms into 
\begin{eqnarray} 
& & u_t=D\Delta u+h(u)-kDu+kw, \nonumber\\ 
& & \tau w_t+\xi u_t=\Delta w, \qquad \qquad \quad \ \qquad \mbox{in $\Omega\times(0,T)$}, \nonumber\\ 
& & \frac{\partial}{\partial \nu}(u,w)=0, \qquad \qquad \qquad \ \ \qquad \mbox{on $\partial\Omega\times(0,T)$}, \nonumber\\ 
& & \left. (u, w)\right\vert_{t=0}=(u_0(x), w_0(x)), \qquad \quad \mbox{in $\Omega$} 
 \label{rdm1} 
\end{eqnarray} 
for $w_0=Du_0+v_0$. In the stationary state we have 
\[ \Delta w=0 \ \mbox{in $\Omega$}, \quad \left. \frac{\partial w}{\partial \nu}\right\vert_{\partial\Omega}=0, \] 
and therefore, this $w$ is a constant denoted by $\overline{w}$. This $\overline{w}$ is prescribed by the initial value using (\ref{totalmass}):  
\begin{equation} 
\tau\overline{w}+\frac{\xi}{\vert\Omega\vert}\int_\Omega u \ dx =\lambda=\frac{1}{\vert\Omega\vert}\int_\Omega \tau w_0+\xi u_0 \ dx. 
 \label{lw}
\end{equation}  
We thus obtain 
\begin{equation} 
-D\Delta u=q(u)+\frac{k}{\tau}\left( \lambda-\frac{\xi}{\vert \Omega\vert}\int_\Omega u \ dx\right), \quad \left. \frac{\partial u}{\partial \nu}\right\vert_{\partial \Omega}=0 
 \label{stationary}
\end{equation} 
for 
\[ q(u)=h(u)-kDu. \] 

The set of stationary solutions to (\ref{rdm1a}), denoted by $E_\lambda$, is thus defined in accordance with $\lambda$ in (\ref{totalmass}), that is, $(u,v)\in E_\lambda$ if and only if $u=u(x)$ is a solution to (\ref{stationary}) for $\xi=1-\tau D$, and $v=\overline{w}-Du$, where $\overline{w}$ is a constant defined by 
\[ \overline{w}=\frac{1}{\tau}\left( \lambda-\frac{\xi}{\vert \Omega\vert}\int_\Omega u \right). \] 

By exploiting the above observations, \cite{mo10, mor12} studied the spectral analysis of the stationary solution. The purpose of the authors in \cite{ls13} was to study the previous results from the point of view of global dynamics.  In fact, with the use of the Lyapunov function, they showed the existence of a global-in-time solution to (\ref{rdm1a}) in $X=C^2(\overline{\Omega})^2$ with compact orbit.  The following theorem is proven by the existence of the Lyapunov functional to (\ref{rdm1a}), 
\[ \frac{d}{dt}\left\{ \xi \int_\Omega \frac{D}{2}\vert \nabla u\vert^2-Q(u) \ dx+\frac{\tau k}{2}\Vert w\Vert_2^2\right\} + \xi\Vert u_t\Vert_2^2+k\Vert \nabla w\Vert_2^2=0. \]

\begin{thm}[\cite{ls13}]\label{thm1}
If $\xi=1-\tau D>0$ it holds that $\omega(u_0,v_0)\subset E_\lambda$. 
\end{thm}

\begin{rmk}
From the result \cite{ftm} established later, the restriction on the space dimension $N=1,2,3$ in \cite{ls13}, is excluded for the compactness of ${\cal O}$. This extension is also valid to the second model described below. 
\end{rmk} 

The problem (\ref{stationary}) has a variational structure. Thus, $u=u(x)$ is a solution if and only if $J_\lambda'(u)=0$, where 
\begin{equation} 
J_\lambda(v)=\int_\Omega \frac{D}{2}\vert \nabla v\vert^2-Q(v)-\frac{k}{\tau}\lambda v \ dx -\frac{k\xi}{2\tau \vert \Omega\vert}\left( \int_\Omega v\right)^2, \quad v\in H^1(\Omega) 
 \label{jei}
\end{equation} 
for $Q'(u)=q(u)$. Then we obtain the dynamical stability of local minimizers of this functional.   

\begin{thm}[\cite{ls13}]\label{thm2}
Let $\xi=1-\tau D>0$ and $h=h(u)$ to be a real-analytic function in $u\in {\bf R}$.  Given $\lambda>0$, let $u_\ast=u_\ast(x)\in H^1(\Omega)$ be a local minimizer of $J_\lambda=J_\lambda(v)$ in (\ref{jei}), and put 
\[ w_\ast=\frac{1}{\tau}\left( \lambda-\frac{\xi}{\vert \Omega\vert}\int_\Omega u_\ast \ dx\right). \] 
Then this stationary solution $(u_\ast,w_\ast)$ to (\ref{rdm1}), for $(u_0, w_0)$ satisfying (\ref{lw}), is dynamically stable in $H^1(\Omega)\times L^2(\Omega)$. Thus, any $\varepsilon>0$ admits $\delta>0$ such that if $(u_0, w_0) \in H^1(\Omega)\times L^2(\Omega)$ satisfies 
\[ \Vert u_0-u_\ast\Vert_{H^1}^2+\Vert w_0-w_\ast \Vert_{2}^2<\delta, \quad \frac{1}{\vert\Omega\vert}\int_\Omega \tau w_0+\xi u_0 \ dx=\lambda, \] 
then it holds that  
\[ \sup_{t\geq 0}\left\{ \Vert u(\cdot, t)-u_\ast\Vert_{H^1}^2+\Vert w(\cdot, t)-w_\ast\Vert_{2}^2\right\}<\varepsilon \] 
for the solution $(u,w)=(u(\cdot, t), w(\cdot,t))$ to (\ref{rdm1}).  
\end{thm}


{\bf Second model.} By letting 
\[ f(u,v)=h(u+v)+\alpha_1v, \quad h(u+v)=-\frac{\alpha_1(u+v)}{(\alpha_2(u+v)+1)^2}, \] 
the model takes the form 
\begin{eqnarray} 
& & u_t=D\Delta u+h(u+v)+\alpha_1v, \nonumber\\ 
& & \tau v_t=\Delta v-h(u+v)-\alpha_1v, \qquad \ \ \mbox{in $\Omega\times(0,T)$}, \nonumber\\ 
& & \frac{\partial}{\partial \nu}(u,v)=0, \qquad \qquad \qquad \qquad \quad \mbox{on $\partial\Omega\times(0,T)$}, \nonumber\\ 
& & \left. (u, v)\right\vert_{t=0}=(u_0(x), v_0(x)), \qquad \quad \ \ \mbox{in $\Omega$}.
 \label{rdm2a} 
\end{eqnarray} 
Since $h=h(z)$ is a smooth function of $z\in {\bf R}$ satisfying  
\begin{equation} 
h(0)=0\geq h(z) \geq -C_2, \quad u\geq 0,   
 \label{eqn:hu}
\end{equation} 
if $(u_0,v_0)\in X=C^2(\overline{\Omega})^2$ with $(u_0, v_0)=(u_0(x), v_0(x))\geq 0$, the problem (\ref{rdm2a}) admits a uniformly bounded unique classical solution $(u,v)=(u(\cdot,t), v(\cdot,t))\geq 0$, global-in-time.  
 
With
\[ w=Du+v, \quad z=u+v, \quad \xi=\frac{1-\tau D}{\tau-1}, \quad \alpha=\frac{1-D}{\tau-1} \] 
and 
\[ g(z)=(1-D)h(z)-\alpha_1Dz \] 
the system (\ref{rdm2a}) transforms into 
\begin{eqnarray} 
& & z_t=D\Delta z+(w_t-D\Delta w+\alpha_1w)+g(z), \nonumber\\ 
& &  w_t+\xi z_t=\alpha\Delta w, \qquad \qquad \quad \ \qquad \mbox{in $\Omega\times(0,T)$}, \nonumber\\ 
& & \frac{\partial}{\partial \nu}(z,w)=0 \qquad \qquad \qquad \ \ \ \qquad \mbox{on $\partial\Omega\times(0,T)$}, \nonumber\\ 
& & \left. (z, w)\right\vert_{t=0}=(z_0(x), w_0(x)), \quad \qquad \mbox{in $\Omega$}, 
 \label{rdm2} 
\end{eqnarray} 
where $(z_0, w_0)=(u_0+v_0, Du_0+v_0)$. The orbit ${\cal O}=\{ (u(\cdot,t), w(\cdot,t))\}_{t\geq 0}$ to this (\ref{rdm2}) is thus compact in $X=C^2(\overline{\Omega})^2$ and hence the $\omega$-limit set defined by 
$$\omega(u_0, w_0)=\{ (u_\ast, w_\ast) \mid \mbox{$\exists t_k\uparrow +\infty $ s.t. $\Vert u(\cdot,t_k)-u_\ast, w(\cdot,t_k)-w_\ast\Vert_X=0$}\}$$
is nonempty, compact, and connected. 

First, total mass conservation arises in the form of  
\begin{equation} 
\frac{1}{\vert \Omega\vert}\int_\Omega \xi z+ w \ dx=\lambda=\frac{1}{\vert \Omega\vert}\int_\Omega \xi z_0+w_0 \ dx.  
 \label{eqn:lambda}
\end{equation} 
Second, there is a Lyapunov functional defined by  
\begin{equation} 
L=L(z,w)=\int_\Omega \frac{\alpha+D}{2}\vert \nabla w\vert^2+\frac{k}{2}w^2+\frac{\xi D}{2}\vert \nabla z\vert^2-\xi G(z) \ dx,   
 \label{eqn:lya1}
\end{equation} 
satisfying 
\begin{equation} 
\frac{dL}{dt} +\xi\Vert z_t\Vert^2_2+\Vert w_t\Vert^2_2+\alpha D\Vert \Delta w\Vert_2^2+\alpha k\Vert \nabla w\Vert_2^2=0.  
 \label{eqn:lya2}
\end{equation} 
Third, in the stationary state of (\ref{rdm2}), the component $w=w(x)$ is spatially homogeneous similarly, denoted by  $w=\overline{w}\in {\bf R}$. Hence it holds that  
\begin{equation} 
\overline{w}=\lambda-\frac{\xi}{\vert\Omega\vert}\int_\Omega z \ dx 
 \label{eqn:ws}
\end{equation} 
by (\ref{eqn:lambda}). Plugging (\ref{eqn:ws}) into the first equation of (\ref{rdm2}), we see that the stationary state of (\ref{rdm1}) is reduced to a single equation concerning $z=z(x)$, that is, 
\begin{equation} 
-D\Delta z=g(z)+k\left(\lambda-\frac{\xi}{\vert \Omega\vert}\int_\Omega z \ dx \right), \quad \left. \frac{\partial z}{\partial\nu}\right\vert_{\partial\Omega}=0. 
 \label{eqn:13}
\end{equation} 
This problem is the Euler-Lagrange equation corresponding to the variational functional 
\begin{equation} 
J_\lambda(z)=\int_\Omega \frac{D}{2}\vert \nabla z\vert^2-G(z)-k\lambda z  \ dx+\frac{k\xi}{2\vert\Omega\vert}\left( \int_\Omega z\ dx\right)^2, \quad z\in H^1(\Omega). 
 \label{eqn:functional}
\end{equation} 
Thus, the set of stationary solutions is associated with $\lambda$ in (\ref{eqn:lambda}), denoted by $E_\lambda$. We say that $(z, \overline{w})\in E_\lambda$, if $z\in H^1(\Omega)$ solves (\ref{eqn:13}) and $\overline{w}\in {\bf R}$ is defined by (\ref{eqn:ws}).

Then we obtain the following results similarly. 
\begin{thm}[\cite{lms15}]\label{thm3}
If $\xi=\frac{1-\tau D}{\tau-1}>0$ it holds that $\omega(z_0,w_0)\subset E_\lambda$. 
\end{thm}

\begin{thm}[\cite{lms15}]\label{thm4}
Let $\xi=\frac{1-\tau D}{\tau-1}>0$ and $h=h(z)$ to be a real-analytic function in $z\in {\bf R}$.  Given $\lambda>0$, let $z\ast =z_\ast(x)\in H^1(\Omega)$ be a local minimizer of $J_\lambda=J_\lambda(z)$ in (\ref{eqn:functional}), and put 
\[ w_\ast=\lambda-\frac{\xi}{\vert \Omega\vert}\int_\Omega z_\ast \ dx. \] 
Then the stationary solution $(z_\ast, w_\ast)$ to (\ref{rdm2}) is dynamical stable in $H^1(\Omega)\times L^2(\Omega)$. Thus, any $\varepsilon>0$ admits $\delta>0$ such that if $(z_0, w_0)\in H^1(\Omega)\times L^2(\Omega)$ satisfies 
\[ \Vert z_0-z_\ast\Vert_{H^1}^2+\Vert w_0-w_\ast \Vert_{2}^2<\delta, \quad \frac{1}{\vert\Omega\vert}\int_\Omega w_0+\xi u_0 \ dx=\lambda, \] 
then it holds that  
\[ \sup_{t\geq 0}\left\{ \Vert z(\cdot, t)-z_\ast\Vert_{H^1}^2+\Vert w(\cdot, t)-w_\ast\Vert_{2}^2\right\}<\varepsilon \] 
for the solution $(z,w)=(z(\cdot, t), w(\cdot,t))$ to (\ref{rdm2}). 
\end{thm} 

\begin{rmk} 
We note the following facts. First, the local minimizer in Theorems \ref{thm2} and \ref{thm4} may be degenerate. Second, there is a correspondence between the Morse index of the linearized operator around the stationary solution $(u_\ast, z_\ast)$ or $(z_\ast, w_\ast)$ and that of $u_\ast$ or $z_\ast$ as a critical point of the variational functional $J_\lambda$.  This property is called the spectral comparison, and a result in this direction is obtained in \cite{lms15} for the second model. 
\end{rmk}

\section{The model and the result}

We skip the third model
\[ f(u,v)=\alpha_1(u+v)[(\alpha u+v)(u+v)-\alpha_2], \] 
because it does not satisfy the quasi-positivity. Hence in this work we consider the fourth model, (\ref{rd0}) for 
\begin{equation} 
f(u,v)=vb\left[\frac{\gamma u^2}{k^2+u^2}+k_0\right]-\delta u, 
  \label{model4}
\end{equation} 
where $\delta>0$. One can also consider more general reaction term used in \cite{he16},  
$$
f(u,v)=vb\left[\frac{\gamma u^m}{k^m+u^m}+k_0\right]-\delta u,\quad m>2 
$$
in the argument below. 

Putting   
\begin{equation} 
a(u)=b\left(\frac{\gamma u^2}{k^2+u^2}+k_0\right), 
 \label{22}
\end{equation} 
we obtain $f(u,v)=va(u)-\delta u$ and 
\begin{equation}\label{a(u)bounds}
a_0\equiv bk_0 \leq a(u)\leq b(\gamma+k_0)\equiv a_1,\quad 0\leq a'(u)\leq b\gamma\alpha(k)
\end{equation}
with
$$
\alpha(k)=\max_{u>0}\frac{2b\gamma k^2u}{(k^2+u^2)^2}=\frac{3\sqrt{3}b\gamma}{16k}.
$$
Therefore, this model is reduced to 
\begin{eqnarray} 
& & u_t=D\Delta u+va(u)-\delta u, \nonumber\\ 
& & \tau v_t=\Delta v-va(u)+\delta u,\qquad \qquad \mbox{in $\Omega\times(0,T),$} \nonumber\\ 
& & \frac{\partial}{\partial \nu}(u, v)=0, \qquad \qquad \qquad \qquad \ \ \mbox{on $\partial\Omega\times(0,T).$} \nonumber\\ 
& & \left. (u, v)\right\vert_{t=0}=(u_0(x), v_0(x)) \qquad \quad \ \ \mbox{in $\Omega$}. 
\label{4thSystem}
\end{eqnarray} 


The nonlinearity $a(u)$ in (\ref{22}) is not so wild. If it is a contact denoted by $a>0$, the system (\ref{4thSystem})  is linear, but a special form of the first model. Hence the stationary state is reduced to 
\begin{equation} 
-D\Delta u=-(\delta+a D)u+\frac{a}{\tau}\left( \lambda-\frac{\xi}{\vert \Omega\vert}\int_\Omega u\right), \quad \left. \frac{\partial u}{\partial \nu}\right\vert_{\partial\Omega}=0 
 \label{24}
\end{equation} 
for $\xi=1-\tau D$ and 
\begin{equation}  
\lambda=\frac{1}{\vert\Omega\vert}\int_\Omega u_0+\tau v_0 \ dx. 
 \label{26x}
\end{equation} 
There is a unique spatially homogeneous solution to (\ref{24}), that is, 
\[ u_\ast=\frac{a\lambda}{a+\tau \delta}. \] 
The linearized operator around this $u_\ast$ is given by 
\[ L\varphi=-D\Delta \varphi+(\delta+aD)\varphi+\frac{a\xi}{\tau\vert\Omega\vert}\int_\Omega \varphi \ dx, \quad \left. \frac{\partial\varphi}{\partial\nu}\right\vert_{\partial\Omega}=0. \] 
 Using the eigenvalues and eigenfunctions of $-\Delta$ under the Neumann boundary condition, we see that this $L$ is non-degenerate always. Thus there is no Turing pattern in this case i.e in the case when $a(u)$ is a constant. 
 
We can actually confirm the linearized stability of this spatially homogeneous stationary solution $(u_\ast, v_\ast)$ to (\ref{4thSystem}) for $v_\ast$ satisfying $\lambda=u_\ast+\tau v_\ast$. In fact, this linearized equation takes the form 
\begin{eqnarray*} 
& & \frac{\partial}{\partial t}\left( \begin{array}{c} 
z \\ 
w \end{array} \right)=\left( \begin{array}{cc} 
D\Delta -\delta & a \\ 
\tau^{-1}\delta & \tau^{-1}\Delta-\tau^{-1}a \end{array} \right) 
\left( \begin{array}{c} 
z \\ 
w \end{array} \right) \quad \mbox{in $\Omega\times (0,T)$} \\ 
& & \left. \frac{\partial}{\partial\nu}(z,w)\right\vert_{\partial\Omega}=0, \quad \frac{1}{\vert \Omega\vert}\int_\Omega z+\tau w  \ dx=0. 
\end{eqnarray*} 
Using the eigenvalues and eigenfunctions of $-\Delta$ under the Neumann boundary condition again, we see that all the eigenvalues of this linearized operator is real and negative. Hence $(u_\ast, v_\ast)$ is asymptotically stable. In spite of these simple profiles of the solution for the the case that $a(u)=a>0$ is a  constant, the global dynamics of (\ref{4thSystem}) for (\ref{22}) is not subject to a Lyapunov functional. 

To confirm this property, we take a look at the stationary problem to \eqref{4thSystem}: 
\begin{eqnarray} 
& & -D\Delta u+\delta u=va(u), \nonumber\\ 
& & -\Delta v=-va(u)+\delta u \quad \mbox{in $\Omega$,} \nonumber\\ 
& & \left.\frac{\partial}{\partial \nu}(u,v)\right\vert_{\partial\Omega}=0, 
 \label{4thSystemStationary}
\end{eqnarray} 
with 
\begin{equation} 
\frac{1}{\vert\Omega\vert}\int_\Omega u+\tau v\;dx=\lambda.
 \label{27}
\end{equation} 
By the argument in the previous section, the function $w=Du+v$ in (\ref{4thSystemStationary}) is a constant denoted by $\overline{w}$, which  is determine by (\ref{27}):  
\[ \overline{w}=\frac{1}{\tau}\left( \lambda-\frac{\xi}{\vert\Omega\vert}\int_\Omega u \ dx\right). \] 
Therefore, the system \eqref{4thSystemStationary} is reduced to 
\begin{eqnarray} 
& & -D\Delta u+\left(\delta+Da(u)\right)u=\frac{a(u)}{\tau}\left(\lambda-\frac{\xi}{\vert\Omega\vert}\int_\Omega u\;dx\right) \ \mbox{in $\Omega$}, \quad \left.\frac{\partial u}{\partial \nu}\right\vert_{\partial\Omega}=0. 
 \label{4thNonlocalEllipticProblem}
\end{eqnarray} 
We see that this (\ref{4thNonlocalEllipticProblem}) admits no variational functional unless $a(u)$ is a constant as in (\ref{24}). Therefore, any Lyapunov function is expected in the non-stationary problem (\ref{4thSystem}). 

The first observation is the existence of a unique spatially homogeneous stationary solution to (\ref{4thSystem}). 

\begin{prop}\label{pro1}
For every $\lambda>0$ there exists a unique $(u_*,v_*)\in {\bf R}^2$ such that 
\begin{equation} 
u_*+\tau v_*=\lambda, \quad f(u_*,v_*)=0, 
 \label{29}
\end{equation} 
for $f=f(u,v)$ defined by (\ref{model4}).
\end{prop}
\begin{proof}
Equality (\ref{29}) is equivalent to 
\[ A(u_\ast)=B(u_\ast), \] 
where 
$$
A(u)=\frac{\tau\delta}{b}+\gamma+k_0+\frac{\lambda\tau\delta}{b(u-\lambda)},\quad B(u)=\frac{\gamma k^2}{k^2+u^2}. 
$$	
The functions $u\in [0,\lambda]\mapsto A(u)$ and $u\in [0,\lambda]\rightarrow B(u)$ are convex and concave, respectively, and hence we obtain the result by 
\[ B(0)<A(0), \quad A(\lambda)=-\infty<B(\lambda).\] 
\end{proof} 

Put $(u_\ast, v_\ast)=(u_\ast(\lambda), v_\ast(\lambda))$ in Proposition \ref{pro1}. The linearized operator around the solution $u_*=u_*(\lambda)$ to (\ref{4thNonlocalEllipticProblem}) is given by  
\begin{eqnarray} 
& &L\varphi = -D\Delta \varphi+(\delta+Da(u_\ast))\varphi+Da'(u_\ast)u_\ast\varphi \\ 
& & \quad -\frac{a'(u_\ast)}{\tau}\left( \lambda-\frac{\xi}{\vert \Omega\vert}\int_\Omega u_\ast \ dx \right) \varphi+\frac{a(u_\ast)\xi}{\tau\vert \Omega\vert}\int_\Omega \varphi\;dx \ \mbox{in $\Omega$}, \quad \left.\frac{\partial \varphi}{\partial \nu}\right\vert_{\partial\Omega}=0. 
 \label{4thNonlocalEllipticProblemEquilibr}
\end{eqnarray} 
We examine the degeneracy of this $L$ in accordance with the eigenvalues $0=\mu_1<\mu_2\leq \cdots \rightarrow\infty$ and the eigenfunctions $\varphi_j$ in $\Vert \varphi_j\Vert_2=1$ for $j=1,2,\cdots$. 

First, for $\mu_1=0$ it hold that $\phi_1=\mbox{constant}$, and this condition is equivalent to 
\[ \delta+Da(u_\ast)+Da'(u_\ast)u_\ast+\frac{a(u_\ast)\xi}{\tau\vert\Omega\vert}\int_\Omega u_\ast \ dx=\frac{a'(u_\ast)}{\tau}\lambda, \] 
although the possible bifurcated object is spatially homogeneous. Second, for $\mu_j$, $j\geq 2$, it holds that  $\int_\Omega\phi_j=0$, and the above degeneracy condition is reduced to 
\[ D\mu_j+\delta+Da(u_\ast)+Da'(u_\ast)u_\ast+\frac{a'(u_\ast)\xi}{\tau\vert \Omega\vert}\int_\Omega u_\ast \ dx=\frac{a'(u_\ast)}{\tau}\lambda. \] 
Then, there is a chance of a spatially inhomogeneous bifurcation.


From the above analysis our main target is revealed. We want to prove that when $D\gg 1$ is the case, in relation to $\lambda$, the solution $(u,v)$ is asymptotically spatially homogeneous. The region that this holds cannot be the entire one because of the possible spatially inhomogeneous bifurcation of stationary states suggested above. Our result in the paper is the following theorem valid under the technical assumption 
\begin{equation} 
N\leq3,  \quad  2\vert \xi\vert a_1<\tau^3(\mu_2D+\delta).  
 \label{technical}
\end{equation} 
Recall that $\mu_2>0$ is the second eigenvalue of $-\Delta$ under the Neumann boundary condition, and $a_1=b(\gamma+k_0)$ is the upper bound of $a(u)$ in (\ref{a(u)bounds}). Note that $\xi=1-\tau D>0$ is not assumed in the following theorem.

\begin{thm}\label{Theorem}
Assume (\ref{technical}). There exists a constant $\sigma=\sigma(b,\gamma, k, k_0, \tau)>0$ such that 
\begin{equation}\label{13}
	\sigma(1+\frac{\lambda^2}{D}) \leq D\mu_2+\delta  
\end{equation}
implies 
\begin{equation} 
\lim_{t\to\infty}\Vert u(\cdot, t)-\overline{u}(t),v(\cdot, t)-\overline{v}(t)\Vert_\infty=0 
 \label{36}
\end{equation} 
for the solution $(u,v)=(u(x,t), v(x,t))$ to (\ref{4thSystem}) for (\ref{22}), where 
\[ \overline{u}(t)=\frac{1}{\vert \Omega\vert}\int_\Omega u(x,t) \ dx, \quad \overline{v}(t)=\frac{1}{\vert\Omega\vert}\int_\Omega v(x,t) \ dx. \] 
The $\omega$-limit set $\omega(u_0, v_0)$ defined by (\ref{omega-limit}) satisfies 
\begin{equation} 
(u_\ast(\lambda), v_\ast(\lambda))\in \omega(u_0, v_0)\subset F_\lambda 
 \label{37}
\end{equation} 
where $(u_\ast(\lambda), v_\ast(\lambda))\in {\bf R}^2$ is the unique solution to (\ref{29}) for $\lambda$ in (\ref{26x}), and 
\begin{equation} 
F_\lambda=\{ (\tilde u_\ast, \tilde v_\ast)\in \overline{\bf R}_+^2 \mid \tilde u_\ast+\tau\tilde v_\ast=\lambda\}. 
 \label{flam}
\end{equation} 
\end{thm}

Since wave-propagation phenomena are reported in numerical simulations \cite{oi07}, \cite{he16}, we can suspect some dynamics inside $\omega(u_0, v_0)$ for the general case.  In accordance with the conclusion (\ref{37}), there is a possibility for $\omega(u_0, v_0)$ to contain the spatially homogeneous orbit of (\ref{4thSystem}). See the final remark of the present paper. Concluding the present section, we refer to \cite{henry} for fundamental concepts on the dynamical systems, $\omega$-limit sets and LaSalle's principle used below. 

\section{Proof of Theorem \ref{Theorem}}
Using $w=Du+v$, we transform (\ref{4thSystem}) into 
\begin{eqnarray} 
& & u_t-D\Delta u+b(u)=a(u)w, \nonumber\\ 
& & \tau w_t+\xi u_t=\Delta w \qquad \qquad \qquad \ \ \mbox{in $\Omega\times (0,T)$} \nonumber\\ 
& & \frac{\partial}{\partial \nu}(u, w)=0 \qquad \qquad \qquad \qquad \mbox{on $\partial\Omega\times (0,T)$} \nonumber\\ 
& & \left. (u, w)\right\vert_{t=0}=(u_0(x),w_0(x)) \qquad \ \mbox{in $\Omega$}, 
 \label{uw} 
\end{eqnarray}  
where 
\[ w_0=Du_0+v_0, \quad \xi=1-\tau D, \quad b(u)=\delta u+Da(u)u. \]

\begin{lem}
	The solution $(u,w)$ to (\ref{uw}) satisfies the estimate:
	\begin{equation}\label{kLemBound}
		\int_0^T(w-\overline{w},\tau w+\xi u-\lambda)\;dt\leq C_3, 
	\end{equation}
	where 
	\[ \overline{w}=\frac{1}{\vert \Omega\vert}\int_\Omega w \ dx. \] 
\end{lem}
\begin{proof}
Recalling 
\[ \tau w+\xi u=u+\tau v, \quad \frac{1}{\vert \Omega\vert}\int_\Omega u+\tau v \ dx=\lambda, \] 
we obtain 
\begin{equation}\label{kLem1}
	(\tau w+\xi u -\lambda)_t-\Delta(w-\overline{w})=0, \quad \left. \frac{\partial w}{\partial \nu}\right\vert_{\partial\Omega}=0
\end{equation}
by (\ref{uw}). By  
\[ \int_\Omega \tau w+\xi u -\lambda \; dx=\int_\Omega \tau w+\xi u -\lambda \;dx=0 \] 
it holds that 
\[ (-\Delta)^{-1}(\tau w+\xi u -c)_t+(w-\overline{w})=0, \] 
where  $(-\Delta)^{-1}f=z$ denotes 
\[ -\Delta z=f \ \mbox{in $\Omega$}, \quad \left. \frac{\partial z}{\partial\nu}\right\vert_{\partial\Omega}=0, \quad \int_\Omega z\;dx=0 \] 
for $f$ satisfying 
\[ \int_\Omega f\;dx=0. \] 
Hence there arises   
\[ \left(
	(-\Delta)^{-1}(\tau w+\xi u -\lambda)_t,\tau w+\xi u-\lambda
\right)
+
\left(
	w-\overline{w},\tau w+\xi u-\lambda
\right)
=0, \] 
where $(\cdot,\cdot)$ is the usual inner product in $L^2(\Omega)$. 

This $(-\Delta)^{-1}$ is a bounded self-adjoint operator in $E=\{ f\in L^2(\Omega) \mid \int_\Omega f \ dx=0\}$, and therefore, it holds that 
\[ 
\frac12\frac{d}{dt}\left(
	(-\Delta)^{-1}(\tau w+\xi u -\lambda),\tau w+\xi u-\lambda
\right)
+
\left(
	w-\overline{w},\tau w+\xi u-\lambda
\right)
=0. 
\] 
Then we obtain (\ref{kLemBound}) because of the positivity of $(-\Delta)^{-1}$.
	\end{proof}
Since 
\[ \lambda=\tau \overline{w}+\xi\overline{u}, \quad \overline{u}=\frac{1}{\vert\Omega\vert}\int_\Omega u \ dx \] 
we get
\begin{equation}\label{wL2bound}
	\tau\int_0^T\|w-\overline{w}\|^2_2\;dt
	\leq
	|\xi|
	\left(
	\int_0^T\|w-\overline{w}\|^2_2\;dt
	\right)^\frac12
		\left(
	\int_0^T\|u-\overline{u}\|^2_2\;dt
	\right)^\frac12
	+C_3
\end{equation}
by (\ref{kLemBound}). Here,  a result on the bounded of $v=w-Du$ follows from the second equation of (\ref{4thSystem}).
\begin{lem}
	If $N\leq3$ there holds that
	\begin{equation}
	\Vert v(\cdot,t)\Vert_2\leq C_4\lambda, \quad t\geq 1. 
	\end{equation}
\end{lem}

\begin{proof}
By 
\[ \tau v_t\leq\Delta v-a_0v+\delta u, \quad \left. \frac{\partial v}{\partial \nu}\right\vert_{\partial\Omega}=0, \quad \left. v\right\vert_{t=0}=v_0(x), \] 
we argue as in \cite{lsy12}, recalling $a_0=b\gamma>0$ in (\ref{a(u)bounds}).  First, we apply the comparison theorem to deduce  
\begin{equation} 
\Vert v(\cdot,t)\Vert_2 \leq \Vert e^{\tau^{-1}(\Delta-a_0)t}v_0\Vert_2
	+\delta\int_0^t\Vert e^{\tau^{-1}(\Delta-a_0)(t-s)}u(s)\Vert_2\;ds, 
 \label{40}
\end{equation} 
where $\Delta$ is provided with the homogeneous Neumann boundary condition. 

Second, the semigroup estimate \cite{rothe} 
\[ \Vert e^{t\Delta}z\Vert_p\leq C_5\max\{ 1, t^{-\frac{n}{2}(\frac{1}{q}-\frac{1}{p})}\} \Vert z\Vert_q, \quad 1\leq q\leq p\leq \infty \] 
is applied to the right-hand side of (\ref{40}). It follows that 
\begin{eqnarray*}
\Vert v(\cdot,t)\Vert_2 & \leq & 
		C_5\{ \Vert v_0\Vert_1
	+
	\delta\int_0^te^{-\tau^{-1}a_0(t-s)}\left(t-s\right)^{-\frac{n}{4}}\Vert u(s)\Vert_1\;ds\\
		&\leq & 
		\lambda C_6\lambda
		\left\{ 1+
	\delta\int_0^\infty e^{-\tau^{-1}a_0t}t^{-\frac{n}{4}}\;ds
	\right\}=C_7\lambda,  
\end{eqnarray*}
provided that $t\geq 1$ and $N\leq 3$. 

\end{proof}


Recall $a_1=b(\gamma+k_0)$ in (\ref{a(u)bounds}). 
\begin{lem}
If 
\begin{equation} 
a_1(1+\frac{1}{2\tau})+\frac{4}{D}(\mu_2^{-1}b\gamma\alpha(k)C_4\lambda)^2\leq \frac{1}{2}(D\mu_2+\delta) 
 \label{if}
\end{equation} 
it holds that 
\begin{equation}\label{ul2distFrw}
\int^T_0\|u-\overline{u}\|_2^2\;dt\leq
\frac{2a_1}{\tau(\mu_2D+\delta)}\int_0^T\|w-\overline{w}\|_2^2\;dt+C_8
\end{equation}
\end{lem}

\begin{proof}
In this proof we use the notations 
\[  \Vert z\Vert^2_2=\frac{1}{\vert\Omega\vert}\int_\Omega z^2\;dx,\quad (z,\zeta)=\frac{1}{\vert\Omega\vert}\int_\Omega z\cdot \zeta\;dx, \quad \overline{z}=\frac{1}{\vert\Omega\vert}\int_\Omega z \ dx.  \] 
Then it follows that 
\[ \quad \Vert z-\overline{z}\Vert^2_2=\Vert z\Vert_2^2-\overline{z}^2, \quad (z, \zeta-\overline{\zeta})=(z-\overline{z}, \zeta-\overline{\zeta}). \] 
	
	We begin by integrating over $\Omega$ the first equation of \eqref{4thSystem} to get:
	$$
	\overline{u}_t+\delta \overline{u}=\overline{a(u)v},
	$$
	and then multiplying with $\overline{u}$
	$$
		\frac12\frac{d}{dt}\overline{u}^2+\delta \overline{u}^2=\overline{a(u)v}\cdot\overline{u}. 
	$$
	Next we test the first equation of \eqref{4thSystem} with $u$:
	$$
	\frac12\frac{d}{dt}\|u\|^2_2+D\|\nabla u\|^2_2+\delta \|u\|_2^2=(a(u)v,u).
	$$
	Subtracting the last two relations above we calculate,
	\begin{eqnarray}
	& & \frac12\frac{d}{dt}\|u-\overline{u}\|^2_2	
	+D\|\nabla u\|^2_2+\delta \|u-\overline{u}\|_2^2=(a(u)v,u)
	-\overline{a(u)v}\cdot\overline{u}
	\nonumber\\
	& & \quad =
	\left(a(u)v,u-\overline{u}\right)=(a(u)v-\overline{a(u)v},u-\overline{u}) \\ 
	& & \quad = (a(u)(v-\overline{v}), u-\overline{u})+(a(u)\overline{v}-\overline{a(u)v},u-\overline{u})  
	 \label{lem9calc}
	\end{eqnarray}
	
	Moreover, we have
	$$
	a(u)\overline{v}-\overline{a(u)v}=a(u)\frac{1}{\vert\Omega\vert}\int_\Omega v\;dy-
	\frac{1}{\vert\Omega\vert}\int_\Omega a(u)v\;dy
	$$
	or
	$$
		a(u(x,t))\overline{v}-\overline{a(u)v}=\frac{1}{\vert\Omega\vert}\int_\Omega \left[a(u(x,t))-a(u(y,t))\right]v(y,t)\;dy, 
	$$
which implies 
\[ \Vert a(u)\overline{v}-\overline{a(u)v}\Vert^2_2 
\leq \Vert v\Vert^2_2\cdot \frac{1}{\vert \Omega\vert^2}\iint_{\Omega\times\Omega} \left[a(u(x,t))-a(u(y,t))\right]^2\;dxdy. \] 
Here, the Poincar\'e-Wirtinger inequality implies 
\begin{eqnarray} 
& & \frac{1}{\vert \Omega\vert^2}\iint_{\Omega\times\Omega}  \vert z(x)-z(y)\vert^2\;dxdy \nonumber\\ 
& & \quad \leq \frac{1}{2}\cdot \frac{1}{\vert \Omega\vert^2}\iint_{\Omega\times \Omega}\vert z(x)-\overline{z}\vert^2+\vert z(y)-\overline{z}\vert^2 \ dxdy \nonumber\\ 
& & \quad =\Vert z-\overline{z}\Vert^2 \leq\mu_2^{-1}\Vert \nabla z\Vert_2^2.
 \label{lem9calc2}
\end{eqnarray}

Therefore, (\ref{lem9calc}) with the help of (\ref{lem9calc2}) becomes,
\begin{eqnarray*}
& & \frac12\frac{d}{dt}\Vert u-\overline{u}\Vert^2_2	
	+D\Vert \nabla u\Vert^2_2+\delta \Vert u-\overline{u}\Vert_2^2	
	\nonumber\\
& & \quad \leq 
	a_1\|u-\overline{u}\|_2\cdot\|v-\overline{v}\|_2
	+\|v\|_2\mu_2^{-1}\|\nabla a(u)\|_2\cdot\|u-\overline{u}\|_2
	\nonumber\\
& & \quad \leq 
	a_1
	\left\{
	\|u-\overline{u}\|_2^2+\tau^{-1}\|w-\overline{w}\|_2\cdot\|u-\overline{u}\|_2
	\right\}
	+b\mu_2^{-1}\gamma\alpha(k)\|v\|_2\|\nabla u\|_2\|u-\overline{u}\|_2
	\nonumber\\
& & \quad \leq a_1\Vert u-\overline{u}\Vert_2^2+\frac{a_1}{2\tau}\Vert w-\overline{w}\Vert_2^2+\frac{a_1}{2\tau}\Vert u-\overline{u}\Vert_2^2+\frac{D}{2}\Vert \nabla u\Vert_2^2 \\ 
& & \qquad +\frac{4}{D}(\mu_2^{-1}b\gamma\alpha(k)C_4\lambda)^2\Vert u-\overline{u}\Vert_2^2 
\end{eqnarray*}
by (\ref{a(u)bounds}), and hence 
\begin{eqnarray*}
& &\frac12\frac{d}{dt}\|u-\overline{u}\|^2_2	
	+(\frac{\mu_2D}{2}+\delta) \|u-\overline{u}\|_2^2	
	\\
& & \quad \leq 
\left[ a_1(1+\frac{1}{2\tau})+\frac{4}{D}(\mu_2^{-1}b\gamma \alpha(k)C_4\lambda)^2\right]\|u-\overline{u}\|^2_2+\frac{a_1}{2\tau}\|w-\overline{w}\|_2^2 
\end{eqnarray*}
by $a_1=b(\gamma+k_0)$. 

Thus, we conclude
$$
	\frac{d}{dt}\|u-\overline{u}\|^2_2	
	+\frac{1}{2}(\mu_2D+\delta) \|u-\overline{u}\|_2^2	
	\leq
\frac{a_1}{\tau}\|w-\overline{w}\|_2^2
$$
from (\ref{if}) and then obtain \eqref{ul2distFrw}.
\end{proof}


Inequality (\ref{if}) arises if we have (\ref{13}) for $\sigma=\sigma(b,\gamma, k, k_0, \tau)>0$ sufficiently large. Then we obtain the following lemma, recalling (\ref{omega-limit}) and (\ref{flam}). 

\begin{lem}\label{lemma6}
Assume (\ref{technical}) and inequality (\ref{13}) for $\sigma>0$ as above. Then it holds that (\ref{36}) and $\omega(u_0,v_0)\subset F_\lambda$.  
\end{lem}

\begin{proof}
	From (\ref{wL2bound}) and (\ref{ul2distFrw}) we get
\begin{eqnarray*} 
& & \tau \int_0^T\Vert w-\overline{w}\Vert_2^2\;dt \\ 
& & \quad \leq \vert \xi\vert
	\left(\int_0^T\Vert w-\overline{w}\Vert_2^2\;dt\right)^\frac12
		\cdot \left(\frac{2a_1}{\tau(\mu_2D+\delta)}\int_0^T\Vert w-\overline{w}\Vert_2^2\;dt+C_8\right)^\frac12+C_3. 
\end{eqnarray*} 
Then (\ref{technical}) implies 
\[ \int_0^T\Vert w-\overline{w}\Vert_2^2\;dt\leq C_9 \] 
and hence 
\begin{equation} 
\int_0^\infty\Vert w-\overline{w}\Vert_2^2\;dt< +\infty. 
 \label{47}
\end{equation} 

From the parabolic regularity and uniformly boundedness of $(u,v)=(u(x,t), v(x,t))$, the mapping $t\to \|w-\overline{w}\|_2^2$ is uniformly continuous, and therefore, we have 
\[ \lim_{t\to\infty}\Vert w-\overline{w}\Vert_2^2=0 \] 
by (\ref{47}). Then, again using the parabolic regularity, we get 
\[ \lim_{t\to\infty}\Vert w-\overline{w}\Vert_\infty=0. \] 

Since (\ref{47}) implies also 
\[ \int_0^\infty\Vert u-\overline{u}\Vert_2^2 \ dt<+\infty \] 
by (\ref{ul2distFrw}), it holds that 
\[ \lim_{t\to\infty}\Vert u-\overline{u}\Vert_\infty=0 \] 
similarly.  We thus obtain (\ref{36}). 

Given $(\hat u_\ast, \hat v_\ast)\in\omega(u_0, v_0)$, therefore, we have $t_k\rightarrow\infty$ such that   
\[ \lim_{k\rightarrow\infty}\Vert \overline{u}(t_k)-\hat u_\ast, \overline{v}(t_k)-\hat v_\ast \Vert_\infty=0 \] 
by (\ref{36}). Therefore, it holds that $(\hat u_\ast, \hat v_\ast)\in F_\lambda$.  
	\end{proof}
	
Now we study the spatially homogeneous part of (\ref{4thSystem}).

	\begin{lem}\label{lemma7}
	 Take $(\tilde u_\ast, \tilde v_\ast)\in \overline{\bf R}_+^2\setminus\{(0,0)\}$ and let $(U,V)=(U(t), V(t))$ be the solution to  (\ref{4thSystem}) for $(u_0, v_0)=(\tilde u_\ast, \tilde v_\ast)$: 
\begin{equation} 
 \frac{dU}{dt}=f(U,V), \ \tau \frac{dV}{dt}=-f(U,V), \quad \left. (U,V)\right\vert_{t=0}=(\tilde u_\ast, \tilde v_\ast). 
\label{ODEeq}
\end{equation} 
Put $\lambda=\tilde u_\ast+\tau \tilde v_\ast>0$ and define $(u_\ast, v_\ast)=(u_\ast(\lambda), v_\ast(\lambda))$ as in Proposition \ref{pro1}. Then it holds that 
\begin{equation}
	\lim_{t\to\infty}(U(t),V(t))=(u_\ast,v_\ast). 
	 \label{eq51}
\end{equation}
	\end{lem}

\begin{proof}
First, we have 
\[ \frac{d}{dt}(U+\tau V)=0 \] 
and hence 
\begin{equation} 
U+\tau V=\lambda. 
 \label{52}
\end{equation} 
Then (\ref{ODEeq}) is reduced to the single system, 
\begin{equation} 
\frac{dU}{dt}=-\delta U+a(U)\tau^{-1}(\lambda-U), \quad \left. U\right\vert_{t=0}=\tilde u_\ast\geq 0 
 \label{41}
\end{equation} 
by (\ref{model4})-(\ref{22}). This system is a spatially homogeneous part of (\ref{4thSystem}), and therefore, there is a global-in-time uniformly bounded orbit ${\cal O}=\{ U(t)\}$ in $\overline{{\bf R}}_+$. 

Finding $G(U)$ such that 
\[ G'(U)=-\delta U+a(U)\tau^{-1}(\lambda-U), \] 
we obtain 
\[ \frac{d}{dt}G(U)=G'(U)\frac{dU}{dt}=[G'(U)]^2\geq 0. \] 
This $-G(U)$ is a Lyapunov function, and therefore, there is 
\[ \lim_{t\rightarrow \infty}G(U(t))=G_\infty \] 
from the compactness of ${\cal O}$. 

The $\omega$-limit set for (\ref{41}) is defined by 
\[ \omega(\tilde u_\ast)=\{ \hat u_\ast \mid \exists t_k\rightarrow \infty \ \mbox{such that} \ \lim_{k\rightarrow \infty}\vert U(t_k)=\hat u_\ast\}. \] 
It is invariant under the flow defined by (\ref{41}). Hence given $\hat u_\ast\in \omega(\tilde u_\ast)$, it holds that 
\[ \tilde U(t)\in \omega (\tilde u_\ast), \quad t\geq 0 \] 
for the solution $\tilde U=\tilde U(t)$ to (\ref{41}) with the initial value $\tilde u_\ast$ replaced by $\hat u_\ast$. It holds also that 
\[ G(\hat u_\ast) = G_\infty, \quad \forall \hat u_\ast \in \omega(\tilde u_\ast), \] 
and hence $G'(U(t))\equiv 0$ (LaSalle's principle). 

This property means that $(U(t), V(t))$ with $V(t)$ defined by (\ref{52}) is a spatially homogeneous stationary solution in Proposition \ref{pro1}, which implies 
\[ \omega(\tilde u_\ast)=\{ u_\ast(\lambda)\}. \] 
Thus we obtain $\lim_{t\rightarrow\infty}U(t)=u_\ast(\lambda)$ ande hence 
$\lim_{t\rightarrow \infty}V(t)=v_\ast(\lambda)$ 
again by (\ref{52}). 
\end{proof}
	Finally, we combine all the above results to prove our main Theorem:
\begin{proof}[Proof of Theorem \ref{Theorem}]
It suffices to show $(u_\ast(\lambda), v_\ast(\lambda))\in \omega(u_0, v_0)$. 

Given $(\tilde u_\ast, \tilde v_\ast)\in \omega(u_0, v_0)\subset F_\lambda$, let $(U,V)$ be the solution to (\ref{4thSystem}) with the initial value $(u_0,v_0)$ replaced by $(\tilde u_\ast, \tilde v_\ast)$. This solution $(U,V)=(U(t), V(t))$ is spatially homogeneous as in (\ref{ODEeq}), and hence it follows that (\ref{eq51}). Then we obtain $(u_\ast(\lambda), v_\ast(\lambda))\in \omega(u_0, v_0)$ from the invariance of the $\omega$-limit set under the flow of (\ref{4thSystem}). 
\end{proof}

\begin{rmk} 
If  $\exists (\tilde u_\ast, \tilde v_\ast)\in \omega(u_0, v_0)\setminus \{ (u_\ast(\lambda), v_\ast(\lambda))\}$ is the case, it follows that $\{ (U(t), V(t))\} \subset \omega(u_0, v_0)$ for $U=U(t)$ and $V=V(t)$ defined by (\ref{41}) and (\ref{52}).   
\end{rmk}


\end{document}